\documentclass[12pt,a4paper,reqno]{amsart}
\usepackage{amsmath}
\usepackage{amsfonts}
\usepackage{amssymb,amsthm}
\numberwithin{equation}{section}

\def\Z{{\mathbb Z}}
\def\E{{\mathbb E}}

\def\R{{\mathbb R}}

\def\P{{\mathbb P}}

\def\1{{\bf 1}}
\def\a{{\mathbf a}}

\def\e{{\mathbf e}}
\def\n{{\mathbf n}}
\def\q{{\mathbf q}}

\def\pmod #1{\ ({\rm{mod}}\ #1)}

\def\cP{{\mathcal P}}
\def\cQ{{\mathcal Q}}

\def\cR{{\mathcal R}}
\def\cS{{\mathcal S}}

\def\cH{{\mathcal H}}
\def\cZ{{\mathcal Z}}
\def\cQ{{\mathcal Q}}
\def\cT{{\mathcal T}}

\def\cE{{\mathcal E}}

\theoremstyle{plain}
\newtheorem{theorem}{Theorem}
\newtheorem{lemma}{Lemma}

\newtheorem{proposition}{Proposition}

\theoremstyle{definition}

\theoremstyle{remark}

\renewcommand{\mod}{\bmod}

\begin{document}

\title
{On the gaps between consecutive primes  }

\author{Yu-Chen Sun}
\address {Medical School, Nanjing
University, Nanjing 210093, People's Republic of China}
\email{b111230069@smail.nju.edu.cn}
\author{Hao Pan}
\address {Department of Mathematics, Nanjing
University, Nanjing 210093, People's Republic of China}
\email{haopan79@zoho.com}



 \begin{abstract}
Let $p_n$ denote the $n$-th prime. For any $m\geq 1$, there exist infinitely many $n$ such that $p_{n}-p_{n-m}\leq C_m$ for some large constant $C_m>0$,
and
$$p_{n+1}-p_n\geq
\frac{c_m\log n\log\log n\log\log\log\log n}{\log\log\log n},
$$
for some small constant $c_m>0$. Furthermore, we also obtain a related result concerning the least primes in arithmetic progressions.
\end{abstract}
\maketitle

\section{Introduction}
\setcounter{lemma}{0}
\setcounter{corollary}{0}
\setcounter{remark}{0}
\setcounter{equation}{0}
\setcounter{conjecture}{0}
\setcounter{proposition}{0}
In recent years, two breakthroughs have been made on the gaps between consecutive primes.
The first one is concerning the small gaps between primes. In 2012, with help of a refinement of the Bombieri-Vinogradov theorem for smooth moduli, Zhang proved that
\begin{equation}\label{Zhang}
\liminf_{n\to\infty}(p_{n+1}-p_n)\leq 7\times 10^7,
\end{equation}
where $p_n$ denotes the $n$-th prime. In 2013, using the multi-dimensional sieve method, Maynard improved the bound for $p_{n+1}-p_n$ to $600$. Furthermore, Maynard and Tao independently proved that for any $m\geq 1$,
\begin{equation}\label{MaynardTao}
\liminf_{n\to\infty}(p_{n+m}-p_n)\leq C_m,
\end{equation}
where $C_m>0$ is a constant only depending on $m$.

The other breakthrough is concerning the large gaps between consecutive primes. In \cite{FGKT16,Maynard16-gap,FGKMT17}, Ford, Green, Konyagin, Maynard and Tao proved that there exist infinitely many $n$ such that
\begin{equation}\label{FGKMTbound}
p_{n+1}-p_n\geq
\frac{c\log n\log_2 n\log_4 n}{\log_3 n},
\end{equation}
where $c>0$ is an absolute constant and $\log_k$ denotes the $k$-th iteration of the logarithum function. This improves the former Erd\H os-Rankin bound
\begin{equation}\label{ERbound}
p_{n+1}-p_n\geq
\frac{c\log n\log_2 n\log_4 n}{\log_3^2 n}.
\end{equation}

In \cite{P-new}, Pintz considered the combination of the large and small gaps between primes. He proved that for any $m\geq1$, there exist infinitely many $n$ such that
\begin{equation}
p_{n}-p_{n-m}\leq C_m
\end{equation}
and
\begin{equation}\label{Pintzlarge}
p_{n+1}-p_n\geq
\frac{c_m\log n\log_2 n\log_4 n}{\log_3^2 n},
\end{equation}
where $C_m,c_m>0$ are two constants only depending on $m$. However, (\ref{Pintzlarge}) still belongs to the Erd\H os-Rankin type bound, which is weaker than (\ref{FGKMTbound}). In this paper, we shall improve Pintz's bound for the large gaps between primes.
\begin{theorem}\label{main}
For any $m\geq 1$, there exist infinitely many $n$ such that
\begin{equation}
p_{n}-p_{n-m}\leq C_m
\end{equation}
and
\begin{equation}
p_{n+1}-p_n\geq
\frac{c_m\log n\log_2 n\log_4 n}{\log_3 n},
\end{equation}
where $C_m,c_m>0$ are two constants only depending on $m$.
\end{theorem}

Let us introduce some necessary notations which will be used frequently later. 
\begin{itemize}
\item
For two quantities $X$ and $Y$, let $X=O(Y)$ or $X\ll Y$ mean that $|X|\leq C|Y|$ for some constant $C>0$. In particular, write $X\asymp Y$ provided that $X=O(Y)$ and $Y=O(X)$ simultaneously.

\item
Suppose that $X$ and $Y$ be  two quantities depending on $x$. Write $X=o(Y)$ provided that $X/Y$ tends to $0$ as $x\to\infty$. In particular, $X\sim Y$ means that $X=Y+o(Y)$ as $x\to\infty$.

\item For an event $A$, let $\P[A]$ denote the probability that $A$ happens. And for another event $B$, define the conditional probability 
$$
\P[A|B]:=\frac{\P[A\wedge B]}{\P[B]}.
$$

\item  For a random variable $X$, let $\E[X]$ the expectation with respect to $X$, i.e.,
$$
\E[X]=\sum_{x}x\cdot P[X=x].
$$
Furthermore, for an event $B$, define the conditional expectation
$$
\E[X|B]:=\sum_{x}x\cdot P\big[(X=x)\wedge B\big].
$$

\item For an assertion $A$, let 
$\1_A=1$ or $0$ according to whether $A$ is true or not.

\end{itemize}

This paper will be organized as follows. First, in the second section, we need to construct a suitable admissible set in order to apply a modified Maynard-Tao theorem. In Section 3, with help of the sieve method, we shall reduce the existence of large gap to the existence of some random vectors. And in the subsequent two sections, we shall complete the proof of Theorem \ref{main} via some probabilistic arguments. Finally, in the last section,  using the similar way, we shall prove a result concerning the lower bound of $p(k,l)$, where $p(k,l)$ denotes the least prime in the arithmetic progression $\{kn+l:\,n\geq 1\}$.

\section{Construction of the admissible set}
\setcounter{lemma}{0}
\setcounter{corollary}{0}
\setcounter{remark}{0}
\setcounter{equation}{0}
\setcounter{conjecture}{0}
\setcounter{proposition}{0}

The following  modified Landau-Page theorem is due to Bank, Freiberg and Maynard.
\begin{lemma}[{\cite[Lemma 4.1]{BFM16}}]\label{MLP}
Suppose that $T\geq 3$ and $P\geq T^{\frac1{\log_2T}}$. Among all $q\leq T$ with $P^+(q)\leq P$ and all primitive characters $\chi$ modulo $q$, there exist at most one $q_0$ and one $\chi$ modulo $q_0$ such that $L(s,\chi)$ has a zero in the region
\begin{equation}\label{sgzero1}
\Re(s)>1-\frac{c_0}{\log P}, \qquad  |\Im(s)|\leq\exp\bigg(\frac{\log P}{\sqrt{\log T}}\bigg),
\end{equation}
where $c_0>0$ is an absolute constant. If such a zero of $L(s,\chi)$ in the region (\ref{sgzero1}) exists, then the zero must be unique, real and simple.
Furthermore, the character $\chi$ is real and
\begin{equation}\label{sgzero2}
P^+(q_0)\gg \log q_0 \gg \log_2T.
\end{equation}
\end{lemma}
Fixed an absolute constant $c$ as in Lemma \ref{MLP} and define
\begin{equation}\label{sgzero3}
\cZ_T=P^+(q_0)
\end{equation}
if such a exceptional modulus $q_0$ exists, and $\cZ_T=1$ if no such modulus exists.

For some integers $h_1,\dots,h_k$, we say that $\cH=\{h_1,\dots,h_k\}$ is admissable, provided that for any prime $p$, there exists $1\leq a\leq p$ such that
$$
h_1,\dots,h_k\not\equiv a\pmod{p}.
$$
\begin{lemma}\label{MT}
Let $m$ be a positive integer. There is a large integer $k_m$ and
a small constant $\epsilon_m>0$ satisfying the following requirements. For any sufficiently large $N$, let $w=\epsilon_m\log N$ and
$$
W=\prod_{\substack{p\leq w\\p\nmid \cZ_{N^{4\epsilon}}}}p
$$
Let $\cH=\{h_1,\dots,h_{k_m}\}$ be an admissible $k_m$-tuple such that
$0\leq h_1,\dots,h_{k_m}\leq N$
and $h_j-h_i$ has no prime factor greater than $w$ for any $1\leq i<j\leq k_m$.
Then for any $1\leq b\leq W$ with $(b+h_1)\cdots(b+h_{k_m})$ and $W$ are co-prime, there exists
$n\in[N,2N]$ with $n\equiv b\pmod{W}$ such that
$
\{n+h_1,\ldots,n+h_{k_m}\}
$
contains at least $m$ primes.
\end{lemma}
\begin{proof}
This is  an immediate corollary of \cite[Theorem 4.2]{BFM16}.
\end{proof}
Suppose that $m\geq 1$ and $N$ is sufficiently large. Let $k_m$ and $\epsilon_m$ be the ones described in Lemma \ref{MT}.
Let
\begin{equation}
R:=\epsilon_m\log N
\end{equation}
and
\begin{equation}
M=\prod_{\substack{p\leq R\\p\nmid \cZ_{N^{4\epsilon}}}}p.
\end{equation}
Let
\begin{equation}\label{defx}
x=\frac{R}{C_0}
\end{equation}
and
\begin{equation}\label{defy}
y=c_0x\cdot\frac{\log x\log_3 x}{\log_2 x},
\end{equation}
where $C_0>2$ is a sufficiently large constant and $c_0>0$ is a sufficiently small constant to be chosen later.

According to the Maynard-Tao theorem, we may choose a large integer $g_m$ (only depending on $k_m$) such that for any admissible $\cH=\{h_1,\dots,h_{g_m}\}$ with $h_i=O_{m}(1)$, there exists $n\in[x,2x]$ such that $\{n+h_1,\ldots,n+h_{g_m}\}$ contains at least $k_m$ primes, where $O_m$ means that the implied constant by $O$ only depends on $m$.
Now suppose that $h'_1,\dots,h'_{k_m}$ are distinct primes lying in $\{n+h_1,\ldots,n+h_{g_m}\}$. It is easy to see that
$$
\cH'=\{h'_1,\dots,h'_{k_m}\}
$$
forms an admissible set.
\begin{proposition}\label{Rz} There exists $1\leq b\leq M$ such that
$$
\big((b+h_1')\cdots(b+h_{m}'),M\big)=1
$$
and
$$
(b+\nu,M)>1
$$
for any $\nu\in(x,y]\setminus\cH'$.
\end{proposition}
Let us explain why Proposition \ref{Rz} implies Theorem. Now $$x\leq h_1',\ldots,h'_{k_m}\leq 2x\leq R.$$ So for any $1\leq i<j\leq k_m$, $h_j'-h_i'$ has no prime factor greater than $R$. By Lemma \ref{MT}, there exists $n\in[N,2N]$  with $$n\equiv b\pmod{M}$$ such that
$$
\{n+h_1',\ldots,n+h_{k_m}'\}
$$
contains at least $m$ primes. Without loss of generality, we may assume that $q_1,\ldots,q_{m'}$ with $q_1<\cdots<q_{m'}$ are all primes lying in $\{n+h_1',\ldots,n+h_{k_m}'\}$, where $m'\geq m$. Clearly
$$
q_{m'}-q_1\leq h_{k_m}'-h_1'=O_m(1).
$$

On the other hand, according to Proposition \ref{Rz}, for any $\nu\in(x,y]\setminus\cH'$, evidently $n+\nu$ can't be a prime since it isn't prime to $M$. Let $q_{m'+1}$ be the least prime greater than $n+y$. Thus $q_1,\ldots,q_{m'},q_{m'+1}$ are consecutive primes such that $q_1,\ldots,q_{m'}$ lie in a bounded interval,  and
$$
q_{m'+1}-q_{m'}\geq y-2x\gg \frac{\log N\log_2 N\log_4 N}{\log_3 N}.
$$

\section{The application of sieve method}
\setcounter{lemma}{0}
\setcounter{corollary}{0}
\setcounter{remark}{0}
\setcounter{equation}{0}
\setcounter{conjecture}{0}
\setcounter{proposition}{0}
Let
\begin{equation}\label{defalpha}
\alpha:=x^{\frac{\log_3 x}{4\log_2 x}}.
\end{equation}
Let
\begin{equation}\label{setS}
\cS:=\{s\text{ prime :}\log^{20} x<s\leq \alpha \}\backslash\{\cZ_{N^{4\epsilon}}\},
\end{equation}
\begin{equation}\label{setP}
\cP:=\{p\text{ prime :}x/2<p\leq x \}\backslash\{\cZ_{N^{4\epsilon}}\},
\end{equation}
and
\begin{equation}\label{setQ}
\cQ:=\{q\text{ prime :}x<q\leq y \}.
\end{equation}
For a prime $p$, let
$$
\cE_p:=\{n\in\Z:\,n\not\equiv h_i'\pmod{p}\text{ for all }1\leq i\leq k_m\}
$$
and
$$
\bar{\cE}_p:=\{a\mod{p}:\,a\not\equiv h_i'\pmod{p}\text{ for all }1\leq i\leq k_m\}.
$$
For residue class $\vec{a}=(a_s\in\bar{\cE}_s)_{s\in \cS}$ and $\vec{b}=(b_p\in\bar{\cE}_p)_{p\in \cP}$, define
$$
S(\vec{a}):=\{ n\in \Z:\,n\not\equiv a_s\pmod{s} \text{ for all }s\in\cS \}
$$
and
$$
S(\vec{b}):=\{ n\in \Z:n\not\equiv b_p\pmod{p} \text{ for all }p\in\cP \}.
$$
\begin{proposition}\label{sievep}
There are vector $\vec{a}=(a_s\in\bar{\cE}_s)_{s\in \cS}$ and $\vec{b}=(b_p\in\bar{\cE}_p)_{p\in \cP}$ such that
\begin{equation}\label{QjSajSb}
\big|\cQ\cap S(\vec{a})\cap S(\vec{b})\big|\ll \frac{x}{\log x}.
\end{equation}
\end{proposition}

Let us explain why Proposition \ref{sievep} implies Theorem \ref{main}.
Let $\vec{a}$ and $\vec{b}$ be as in Proposition \ref{sievep}. We extend the tuple $\vec{a}=(a_s)_{s\in\cS}$ to $(a_p)_{p\in[1,x]\setminus\{\cZ_{N^{4\epsilon}}\}}$, by setting $a_p:=b_p$ for $p\in \cP$, $a_p:=0$ for $p\in[1,x]\setminus(\cS\cup\cP\cup\{\cZ_{N^{4\epsilon}}\})$. Let
$$
\cT:=\big\{ n\in(x,y]:n\not\equiv a_p\pmod{p}\text{ for all }p\in[1,x]\setminus\{\cZ_{N^{4\epsilon}}\} \big\}.
$$
For any prime $q\in\cQ\cap S(\vec{a})\cap S(\vec{b})$, since $q>x$, clearly $$q\not\equiv 0=a_p\pmod{p}$$
for each $p\in[1,x]\setminus(\cS\cup\cP\cup\{\cZ_{N^{4\epsilon}}\})$. It follows that
$$
\cQ\cap S(\vec{a})\cap S(\vec{b})\subseteq\cT.
$$
Suppose that $n\in \cT\setminus(\cQ\cap S(\vec{a})\cap S(\vec{b}))$. Then
$n$ is composite and no prime factor of $n$ lies in $([1,\log^{20} x]\cup(\alpha,x/2])\setminus \{\cZ_{N^{4\epsilon}}\}$.
Also, since $y/x=o(\log x)$ and recalling that
$$
\cZ_{N^{4\epsilon}}\gg\log\log N\gg\log x,
$$
we obtain that $n$ is not divisible by any prime lying in $(x/2,y]$. It follows that all prime factors of $n$ all lie in  $(\log^{20} x,\alpha]$.
Let $\cR$ be the set of $\alpha$-smooth numbers, i.e., a number with all prime factors at most $\alpha$. By the well-known de Brujin's theorem \cite{deB51} and (\ref{defy}), we have
$$
\big|\cR\cap[1,y]\big|= e^{-u\log u+O(u\log\log(u+2))}=\frac{y}{\log^{4+o(1)}x}=o\bigg(\frac{x}{\log x}\bigg),
$$
where
$$
u:=\frac{\log y}{\log \alpha}\sim \frac{4\log_2 x}{\log_3 x}
$$
by (\ref{defy}) and (\ref{defalpha}).
Hence $$
\big|\cT\setminus\big(\cQ\cap S(\vec{a})\cap S(\vec{b})\big)\big|=o\bigg(\frac{x}{\log x}\bigg),$$
i.e.,
$$
|\cT|=O\bigg(\frac{x}{\log x}\bigg)
$$
by (\ref{QjSajSb}).

Now let $C_0$ in (\ref{defx}) be a sufficiently large constant such that $$
|\cT|\leq\frac{(C_0-1)x}{3\log x},
$$
i.e., $2|\cT|$ is less than the number of primes in $(x,C_0x]$. Assume that $\cT\setminus\cH'=\{t_1,\ldots,t_h\}$. Note that for any $1\leq i\leq h$ and $1\leq j\leq k_m$, $t_i-h_j'$ has at most one prime factor lying in $(x,C_0x]$, since $t_i\leq y<x^2$. Choose  distinct primes $p_1,\ldots,p_h$ in $(x,C_0x]$ such that
$t_i\not\equiv h_j'\pmod{p_i}$ for any $1\leq i\leq h$ and $1\leq j\leq k_m$. Set $a_{p_i}=t_i$ for each $1\leq i\leq h$. Then $a_{p_i}\in\bar{\cE}_{p_i}$. Also, for each $p\in (x,C_0x]\setminus\{p_1,\ldots,p_h\}$, we may arbitrarily choose $a_p\in\bar{\cE}_p$.
Recall that $C_0x=R$ and $M$ is the product of all primes in $[1,R]\setminus\{\cZ_{N^{4\epsilon}}\}$. By the Chinese remainder theorem, there exists $1\leq b\leq M$ such that
$$
b\equiv -a_{p}\pmod{p}
$$
for each prime $p\in [1,C_0x]\setminus\{\cZ_{N^{4\epsilon}}\}$. We claim that $b$ satisfies the requirements of Theorem \ref{main}. It is easy to see that $a_p\in\bar{\cE}_{p}$ for any $p\in [1,C_0x]\setminus\{\cZ_{N^{4\epsilon}}\}$. So we have $(b+h_1')\cdots(b+h_{k_m}')$ is prime to $M$. Also, suppose that $\nu\in(x,y]\setminus\cH'$. If $\nu\not\in\cT$, then $\nu\equiv a_p\pmod{p}$ for some $p\in [1,x]\setminus\{\cZ_{N^{4\epsilon}}\}$, i.e., both $b+\nu$ and $M$ are divisible by $p$. If $\nu\in\cT\setminus\cH'$, then we also have $\nu=t_i=a_{p_i}$ for some $1\leq i\leq h$. So we always have $(b+\nu,M)>1$ for any $\nu\in(x,y]\setminus\cH'$. Thus Theorem \ref{main} is concluded.

In fact, according to the discussions of \cite{FGKMT17}, Proposition\ref{sievep} can be deduced from the following weaker random construction.
\begin{proposition}\label{RanCon}
Suppose that $x$ is sufficiently large and $y$ is the one in (\ref{defy}). Then there is a quantity $C$ with $C\asymp c_0^{-1}$
, a tuple of positive integers $(v_1,\dots,v_r)$ with $r\leq\sqrt{\log x}$ and some way to choose random vector $\vec{\a}=(\a_s\not\equiv h'_i\mod s)_{s\in\cS}$ and $\vec{\n}=({\n_p\in\cE_p})_{p\in\cP}$ of congruence classes $\a_s\mod s$ and integers $\n_p$ respectively, satisfying the following requirements:

\medskip\noindent
(I) For every $\vec{a}$ in the essential range of $\vec{\a}$, we have
\begin{equation}\label{Pqine}
\P\big[q\in\e_p(\vec{a})\big|\,\vec{\a}=\vec{a}\big]\leq x^{-\frac12-\frac1{10}}
\end{equation}
for each $p\in\cP$, where $\e_p(\vec{a}):=\{\n_p+v_ip:1\leq i\leq r\}\cap\cQ\cap S(\vec{a})$.\\

\medskip\noindent(II) With probability $1-o(1)$, we have
\begin{equation}\label{QjSa}
\#\big(\cQ\cap S(\vec{a})\big)\sim 80c\cdot\frac{x\log_2{x}}{\log x}.
\end{equation}

\medskip\noindent(III) Suppose that $\vec{a}$ lies in the essential range of $\vec{\a}$. We call $\vec{a}$ good, provided that for all but at most $x/(\log x\log_2 x)$ elements $q\in\cQ\cap S(\vec{a})$,
\begin{equation}\label{RanConSumP}
\sum_{p\in\cP}\P\big[q\in\e_p(\vec{a})\big|\,\vec{\a}=\vec{a}\big]=C+O_{\leq}\bigg(\frac{1}{(\log_2 x)^2}\bigg),
\end{equation}
where $O_\leq$ means that the implied constant by $O$ is just $1$.
Then with probability $1-o(1)$, $\vec{a}$ is good.
\end{proposition}
The proof of Proposition \ref{RanCon} will be given in the next section, via a number of auxiliary lemmas.

\section{Proofs of (I) and (II) of Proposition \ref{RanCon}}
\setcounter{lemma}{0}
\setcounter{corollary}{0}
\setcounter{remark}{0}
\setcounter{equation}{0}
\setcounter{conjecture}{0}
\setcounter{definition}{0}
\setcounter{proposition}{0}

\begin{lemma}\label{SieveWeight}
Let $x$ and $y$ be the sufficiently large real numbers defined by (\ref{defx}) and (\ref{defy}). Let $\cP,\cQ$ be defined by (\ref{setP}), (\ref{setQ}). Let $r$ be a positive integer with
\begin{equation}\label{SWr}
r_0\leq r\leq \log^{\eta_0}x
\end{equation}
for some sufficiently large absolute constant $r_0$ and sufficiently small $\eta_0>0$, let $(v_1,\dots,v_r)$ be an admissible $r$-tuple contained in $[2r^2]$. Then one can find a positive integer quantity
\begin{equation}\label{SWtau}
\tau\geq x^{-o(1)}
\end{equation}
and a positive quantity $u=u(r)$ depending only on $r$ with
\begin{equation}\label{SWu}
u\asymp\log r
\end{equation}
and a non-negative $w:\cP\times\Z\to\R^+$ supported on $\cP\times[-y,y]$ with the following properties:
\begin{itemize}
\item Uniformly for every $p\in\cP$, one has
\begin{equation}\label{SWsumWpn}
\sum_{n\in\Z}w(p,n)=(1+O(\frac{1}{\log^{10}_2x}))\tau\frac{x}{\log^r x}
\end{equation}
\item Uniformly for every $q\in\cQ$ and $i=1,\dots,r$ one has
\begin{equation}\label{SWsumWpqhp}
\sum_{p\in\cP}w(p,q-v_ip)=(1+O(\frac{1}{\log^{10}_2x}))\tau\frac{u}{r}\frac{x}{2\log^r x}
\end{equation}
\item Uniformly for every $v=O(y/x)$ that is not equal to any of the $v_i$, one has
\begin{equation}\label{SWsumsumWpqhp}
\sum_{q\in\cQ}\sum_{p\in\cP}w(p,q-vp)=O(\frac{1}{\log_2^{10}x}\tau\frac{x}{\log^r x}\frac{y}{\log x})
\end{equation}
\item Uniformly for all $p\in\cP,n\in\Z$
\begin{equation}\label{SWWpn}
w(p,n)=O(x^{1/3+o(1)})
\end{equation}
\end{itemize}
\end{lemma}
\begin{proof}
See \cite[Theorem 6]{FGKMT17}.
\end{proof}
 Let $x,y,c,\cS,\cP$ and $\cQ$ be as in Proposition \ref{RanCon}. Set
 \begin{equation}\label{maxr}
 r:=\lfloor \log^{\eta_0} x \rfloor,
 \end{equation}
where $\eta_0$ is the small constant in Lemma \ref{SieveWeight}.
Let $w:\cP\times\Z\to\R^+$ be a weight satisfying all stated properties in Lemma \ref{SieveWeight}.

For each $p\in\cP$, let $\tilde{\n}_p$ denote the random integer with probability density
$$
\P[\tilde{\n}_p=n]:=\frac{w(p,n)}{\sum_{n'\in\Z}w(p,n')}.
$$
It follows that
$$
\P[\tilde{\n}_p=q-v_ip]:=\frac{w(p,q-v_ip)}{\sum_{n'\in\Z}w(p,n')}.
$$
In view of (\ref{SWsumWpn}) and (\ref{SWsumWpqhp}), we have
\begin{equation}\label{sumPqnpvip}
\sum_{p\in\cP}\P[q=\tilde{\n}_p+v_ip]=\bigg(1+O\bigg(\frac{1}{\log^{10}_{2} x}\bigg)\bigg)\cdot\frac{u}{r}\cdot\frac{x}{2y}
\end{equation}
for any $q\in\cQ$ and $1\leq i\leq r$.
Also, by (\ref{SWsumWpn}), (\ref{SWWpn}), (\ref{SWtau}) and (\ref{defy}), we get that
\begin{equation}\label{Pnpn}
\P[\tilde{\n}_p=n]\ll x^{-\frac12-\frac16+o(1)}
\end{equation}
for each $p\in\cP$ and $n\in\Z$.

Let the random vector $\vec{\a}:=(\a_s)_{s\in\cS}$ be chosen by selecting each $\a_s\in\Z_s$ uniformly at random from $\Z_s\backslash\{h'_i\mod{s}:\,i=1,\dots,m\}$, independently in $s$ and independently of the $\tilde{\n}_p$.
Let $h(s)$ denote the number of residue classes $h'_i\mod s$, $1\leq i\leq m$.
Since $\{h_1',\ldots,h_{k_m}'\}$ is admissible, we always have $h(s)<s$.
Then $S(\vec{\a})$ is a random periodic
subset of $\Z$ with density
$$
 \prod_{s\in\cS}\bigg(1-\frac{1}{s-h(s)}\bigg)\sim\sigma,
$$
where
$$
\sigma:=\prod_{s\in\cS}\bigg(1-\frac{1}{s}\bigg).
$$
According to (\ref{defalpha}), (\ref{setS}) and the prime number theorem (with a sufficiently strong error term), we have
$$
\sigma=\bigg(1+O\bigg(\frac{1}{\log^{10}_2 x}\bigg)\bigg)\cdot\frac{\log(\log^{20}x)}{\log \alpha}=\bigg(1+O\bigg(\frac{1}{\log^{10}_{2} x}\bigg)\bigg)\cdot\frac{80(\log_2 x)^2}{\log x \log_3 x}
$$
In particular, by (\ref{defy}),
\begin{equation}\label{Sigma_y}
\sigma y=\bigg(1+O\bigg(\frac{1}{\log^{10}_{2} x}\bigg)\bigg)\cdot80cx\log_2 x
\end{equation}
We also obtain from (\ref{maxr}) that
\begin{equation}\label{SigmaPower_r}
\sigma^{r}=x^{o(1)}.
\end{equation}

For each $p\in\cP$, let
\begin{equation}\label{Xpa}
X_p(\vec{a}):=\P\big[\tilde{\n}_p\in\cE_{p}\text{ and }\tilde{\n}_p+v_ip\in S(\vec{a})\text{ for any }1\leq i\leq r\big]
\end{equation}
and let $\cP(\vec{a})$ denote the set of all the primes $p\in\cP$ such that
\begin{equation}\label{Xpabound}
X_p(\vec{a})=\big(1+O_{\leq}(\log^{-3} x)\big)\cdot\sigma^r.
\end{equation}
For each $p\in\cP$, define the random variables $\n_p$ as follows. Suppose that the event $\vec{\a}=\vec{a}$ happens for
some $\vec{a}$ in the range of $\vec{\a}$. Set $\n_p = 0$ if $p\in\cP\backslash\cP(\vec{a})$. If $p\in\cP(\vec{a})$, let $\n_p$ to be the random integer with conditional probability distribution
\begin{equation}\label{PZ}
\P[\n_p=n|\,\vec{\a}=\vec{a}]:=\frac{Z_{p}(\vec{a};n)}{X_p(\vec{a})},
\end{equation}
where
$$
Z_p(\vec{a};n)=\1_{\substack{n+v_ip\in S(\vec{a})\\ n\in\cE_{p}}}\cdot\P[\tilde{\n}_p=n].
$$
Note that here those random variables $\n_p$ for all $p\in\cP(\vec{a})$ are jointly independent and conditionally on the event $\vec{\a}=\vec{a}$.
 From (\ref{Xpa}) we see that these
random variables are well defined.

Clearly the first assertion of Proposition \ref{RanCon} easily follows from  (\ref{Pnpn}), (\ref{PZ}) and (\ref{Xpa}). 
\begin{lemma}\label{FGKMTcorrbound}
Let $t\leq \log x$ be a natural number with $|n_1|,\ldots,|n_t|\leq y$. Then
$$
\P\big[n_1,\dots,n_t\in S(\vec{\a})\big]=\big(1+O(\log^{-16} x)\big)\cdot\sigma^t.
$$
\end{lemma}
\begin{proof}
Clearly
$$
\P[\a_s\not\equiv  n_1,\dots,n_t\mod s]\geq\frac{s-h(s)-t}{s-h(s)}.
$$
For each $s\in\cS$, clearly the integers $n_1,\dots,n_t$ will occupy $t$ distinct residue classes modulo $s$, unless $s$ divides $n_i-n_j$ for some $1\leq i < j \leq t$. Since $s \geq \log^{20} x$ and $|n_i-n_j|\leq 2x^2$, we have
$$
\big|\{s\in\cS:\,n_i\equiv n_j\pmod{s}\text{ for some }1\leq i<j\leq t\}\big|=O(t^2 \log x) = O(\log^{3} x).
$$
Similarly,
$$
\big|\{s\in\cS:\,n_i\equiv h_j'\pmod{s}\text{ for some }1\leq i\leq t\text{ and }1\leq j\leq m\}\big|= O(m\log^{2} x).
$$
Thus except for $O(\log^{3} x)$ values of $s$,
$$
\P[\a_s\not\equiv  n_1,\dots,n_t\mod s]=\frac{s-h(s)-t}{s-h(s)}.
$$
Since the choice for those components $a_s$ are independent, we have
\begin{align*}
\P\big[n_1,\dots,n_t\in S(\vec{\a})\big]= &
\bigg(1+O\bigg(\frac{t}{s}\bigg)\bigg)^{O(\log^3 x )}
\prod_{s\in \cS}\bigg(1-\frac{t}{s-h(s)}\bigg).
\end{align*}
Note that
\begin{align*}
\prod_{s\in \cS}\bigg(1-\frac{t}{s-h(s)}\bigg)\bigg(1-\frac{1}{s}\bigg)^{-t}=&\prod_{s\in \cS}\exp\bigg(O\bigg(\frac{t^2+mt}{s^2}\bigg)\bigg)\\
=&1+O\bigg(\sum_{s\in\cS}\frac{t^2+mt}{s^2}\bigg)=1+O\bigg(\frac{1}{\log^{18}x}\bigg).
\end{align*}
It follows that
\begin{align*}
\P\big[n_1,\dots,n_t\in S(\vec{\a})\big]=&
\bigg(1+O\bigg(\frac{t\log^3 x}{s}\bigg)\bigg)\cdot \bigg(1+O\bigg(\frac{1}{\log^{18}x}\bigg)\bigg)
\prod_{s\in \cS}\bigg(1-\frac{1}{s}\bigg)^{t}\\
=&\big(1+O(\log^{-16} x)\big)\cdot\sigma^t.
\end{align*}
\end{proof}

Let us introduce some basic inequalities in Probability Theory
\begin{itemize}
\item Markov's inequality
\begin{equation}\label{Markov}
\P[X\geq \lambda]\leq\frac{\mu}{\lambda},\quad X>0,\mu= \E[X],\lambda>0
\end{equation}
\item Chebyshev's inequality
\begin{equation}\label{Chebyshev}
\P[|X>\mu|\geq\lambda\sqrt{\E|X-\mu|^2}]\leq\frac{1}{\lambda^2},\quad \lambda>0,\quad\mu=\E X\in\R,\quad\E|X-\mu|^2>0
\end{equation}
\end{itemize}

\begin{lemma}[modified Chebyshev's inequality]\label{Mchebyshev}
Suppose that for some $A > 0$ and $0 < \epsilon < 1$ we have
$$
\mu=\E[X]=A(1+O_{\leq}(\epsilon)),\quad \E X^2=A^2(1+O_{\leq}(\epsilon)).
$$
Then, for any $\delta>\epsilon $ we have
$$
\P[|X-A|\geq\delta A]\leq\frac{4\epsilon}{(\delta-\epsilon)^2}
$$
\end{lemma}
\begin{proof} See \cite[Lemma 2.1]{FGKMT17}.
\end{proof}

\begin{lemma}\label{estQjSa}
With the probability $1-o(1)$, we have
\begin{equation}\label{QjSabound}
\big|\cQ\cap S(\vec{\a})\big|=\big(1+O(\log_2^{-10} x)\big)\cdot\frac{\sigma y}{\log x}.
\end{equation}
\end{lemma}
\begin{proof}
Applying Lemma \ref{FGKMTcorrbound} with $t=1$ and $2$, we can respectively get that
$$
\E\big[|\cQ\cap S(\vec{\a})|\big]=\big(1+O(\log^{-16} x)\big)\sigma\cdot|\cQ|
$$
and
$$
\E\big[|\cQ\cap S(\vec{\a})|^2\big]=\big(1+O(\log^{-16} x)\big)\cdot\big(\sigma\cdot |\cQ|+\sigma^2\cdot|\cQ|(|\cQ|-1)\big).
$$
By Lemma \ref{Mchebyshev}, we have
$$
\P\big[\big||\cQ\cap S(\vec{a})|-\sigma|\cQ|\big|\geq\log_2^{-10} x\sigma|\cQ|\big]=o\big(\log^{-16} x\big).
$$
And $|\cQ|\sim y\log^{-1} x$ by the prime number theorem. Thus the proof of Lemma \ref{estQjSa} is complete.
\end{proof}
Thus (II) of Proposition \ref{RanCon} is an immediate consequence of (\ref{defy})  and (\ref{QjSabound}).

\section{Proof of (III) of Proposition \ref{RanCon}}
\setcounter{lemma}{0}
\setcounter{corollary}{0}
\setcounter{remark}{0}
\setcounter{equation}{0}
\setcounter{conjecture}{0}
\setcounter{definition}{0}
\setcounter{proposition}{0}

\begin{lemma}\label{FGKMTlem63}
With the probability $1-O(\log^{-3} x)$, $\cP(\vec{\a})$ contains all but $O(x\log^{-4} x)$ of the primes $p\in\cP$. In particular, $$\E\big[|\cP(\vec{\a})|\big]=|\cP|\cdot\big(1+O(\log^{-3}x)\big).$$
\end{lemma}
\begin{proof}
By Lemma \ref{FGKMTcorrbound}, we have
\begin{align*}
\E\big[X_p(\vec{\a})\big]=&\sum_{n_p}\E\big[X_p(\vec{\a})|\,\tilde{\n}_p=n_p\big]\cdot\P[\tilde{\n}_p=n_p]\\
=&\big(1+O(\log^{-16} x)\big)\cdot \sigma^r\sum_{n_p\in\cE_{p}}\P[\tilde{\n}_p=n_p].
\end{align*}
And according to (\ref{Pnpn}),
$$
\sum_{n_p\in\cE_{p}}\P[\tilde{\n}_p=n_p]=1-O\bigg(\frac{y}{p}\bigg)\cdot x^{-\frac12-\frac16+o(1)}=1-O\big(x^{-\frac12-\frac16+o(1)}\big).
$$
Hence we get that
\begin{equation}\label{EXp}
\E\big[X_p(\vec{\a})\big]=\P[\tilde{\n}_p+v_ip\in S(\vec{\a})\text{ for each }1\leq i\leq r]=\big(1+O(\log^{-16} x)\big)\cdot \sigma^{r}.
\end{equation}
Similarly, by Lemma \ref{FGKMTcorrbound}, we also have
$$
\E\big[X_p(\vec{\a})^2\big]=\big(1+O(\log^{-16} x)\big)\sum_{n^{(1)}_p,n^{(2)}_p\in\cE_{p}}\sigma^{|R_{n^{(1)}_p,n^{(2)}_p}|}\P[\tilde{\n}_p^{(1)}=n^{(1)}_p\text{ and }\tilde{\n}_p^{(2)}=n^{(2)}_p],
$$
where $\tilde{\n}_p^{(1)}, \tilde{\n}_p^{(2)}$ are independent copies of $\n_p$  that are also independent of $\vec{\a}$, and
$$
R_{n^{(1)}_p,n^{(2)}_p}=\{\n^{(l)}_p+v_ip:\,1\leq i\leq r,\ l=1,2\}.
$$
Evidently $|R_{n^{(1)}_p,n^{(2)}_p}|\neq 2r$ if and only if $n^{(1)}_p-n^{(2)}_p=(v_i-v_j)p$ for some $1\leq i,j\leq r$. By (\ref{Pnpn}),
\begin{align*}
&\sum_{\substack{n^{(1)}_p,n^{(2)}_p\in\cE_{p}\\
|R_{n^{(1)}_p,n^{(2)}_p}|\neq 2r
}}\P[\tilde{\n}_p^{(1)}=n^{(1)}_p\text{ and }\tilde{\n}_p^{(2)}=n^{(2)}_p]\\
=&\sum_{\substack{n^{(1)}_p,n^{(2)}_p\in\cE_{p}\\
1\leq j,k\leq r}}
\P[\tilde{\n}_p^{(1)}=n^{(1)}_p]\P\big[\tilde{\n}_p^{(2)}=n^{(1)}+(v_j-v_k)p\big]\\
=&O(x^{-\frac12-\frac16+o(1)}).
\end{align*}
While
\begin{align*}
&\sum_{n^{(1)}_p,n^{(2)}_p\in\cE{p}}\P[\tilde{\n}_p^{(1)}=n^{(1)}_p\text{ and }\tilde{\n}_p^{(2)}=n^{(2)}_p]\\
=&\bigg(\sum_{n_p\in\cE_{p}}\P[\tilde{\n}_p=n_p]\bigg)^2
=1-O(x^{-\frac12-\frac16+o(1)}\big).
\end{align*}
Since $\sigma^r=x^{o(1)}$, we get
\begin{equation}\label{VarXp}
\E\big[X^2_p(\vec{\a})\big]=\big(1+O(\log^{-16} x)\big)\cdot\sigma^{2r}.
\end{equation}

For each $p\in\cP$, from the (\ref{Xpa}), we know that
$$
p\in \cP(\vec{a})\Longleftrightarrow |X_p(\vec{a})-\sigma^r|\leq O_{\leq}\big(\log^{-3} x\big)\cdot\sigma^r,
$$
Thus by Lemma \ref{Mchebyshev}, we have
$$
\P\big[|X_p(\vec{\a})-\sigma^r|\leq O_{\leq}(\log^{-3} x)\cdot \sigma^r)\big]=1-O\big(\log^{-6} x\big).
$$
That is, $p$ belongs to $\cP(\vec{\a})$ with probability $1-O(\log^{-6} x)$ for each $p\in\cP$.
Then with help of the linearity of expectation, we have
$$
\E\big[|\cP(\vec{\a})|\big]=\E\bigg[\sum_{p\in\cP}\1_{p\in\cP(\vec{\a})}\bigg]=\sum_{p\in\cP}\E\big[\1_{p\in\cP(\vec{\a})}\big]=|\cP|\cdot\big(1+O(\log^{-6} x)\big).
$$
Since $\cP(\vec{\a})\subseteq\cP$ and $|\cP|\asymp x/\log x$, by Markov's inequality
$$
\P\bigg[|\cP|-|\cP(\vec{\a})|\geq \frac{x}{\log^{4} x}\bigg]\leq \frac{\E(|\cP|-|\cP(\vec{\a})|}{x\log^{-4} x}=O\bigg(\frac{1}{\log^{3} x}\bigg).
$$
\end{proof}
\begin{lemma} For each $1\leq i\leq r$,
\begin{equation}\label{FGKMTpf62620}
\E\bigg[\sum_n\sigma^{-r}\sum_{p\in\cP\backslash\cP(\vec{\a})}Z_p(\vec{\a};n)\bigg]=o\bigg(\frac{u}{\sigma}\cdot\frac{x}{2y}\cdot\frac{1}{r\log^3_2x\log^2_2 x}\bigg).
\end{equation}
\end{lemma}
\begin{proof}
First, it follows from Lemma \ref{FGKMTcorrbound} that
\begin{align*}
\E\bigg[\sum_n\sigma^{-r}\sum_{p\in\cP}Z_{p}(\vec{\a};n)\bigg]=&\sigma^{-r}\sum_{p\in\cP}\sum_{n\in\cE_{p}}\P[\tilde{\n}_p=n]\P\big[n+v_ip\in S(\vec{\a})\big]\\
=& \big(1+O(\log^{-16} x)\big)\cdot|\cP|.
\end{align*}
Next, by (\ref{Xpa}) and Lemma \ref{FGKMTlem63}, we have
\begin{align*}
&\E\bigg[\sum_n\sigma^{-r}\sum_{p\in\cP(\vec{\a})}Z_{p}(\vec{\a};n)\bigg]=\sigma^{-r}\sum_{\vec{a}}\P[\vec{\a}=\vec{a}]\sum_{p\in\cP(\vec{a})}X_p(\vec{a})\\
= &\big(1+O(\log^{-3} x)\big)\cdot\E\big[|\cP(\vec{\a})|\big]=\big(1+O(\log^{-3} x)\big)\cdot|\cP|.
\end{align*}
So we conclude that the left-hand side of (\ref{FGKMTpf62620}) is $$O\big(|\cP|\cdot\log^{-3} x\big)=O\big(x\log^{-4} x\big).$$ Then  (\ref{FGKMTpf62620}) immediately follows form (\ref{defy}) and (\ref{SWr}).
\end{proof}
\begin{lemma}\label{sigam-rsumsumP} With the probability $1-o(1)$,
\begin{equation}\label{Zp}
\sum_{i=1}^{r}\sum_{p\in\cP}Z_{p}(\vec{\a};q-v_ip)=\big(1+O_{\leq}(\log_2^{-3} x)\big)\cdot\sigma^{r-1}u\cdot\frac{x}{2y},
\end{equation}
for all but at most $x/(2\log x\log_2 x)$ primes $q\in\cQ\cap S(\vec{\a})$.
\end{lemma}
\begin{proof}
Call  a prime $q\in\cQ$ {\it bad}, if $q\in\cQ\cap S(\vec{\a})$ but (\ref{Zp}) fails.
With help of Lemma \ref{FGKMTcorrbound} and (\ref{sumPqnpvip}), we have
\begin{align*}
& \E\bigg[\sum_{q\in\cQ\cap S({\vec{\a}})}\sum^{r}_{i=1}\sum_{p\in\cP}Z_{p}(\vec{\a};q-v_ip)\bigg]\\ = & \sum_{\substack{q,i,p\\q-v_ip\in\cE_{p}}}\P\big[q+(v_j-v_i)p\in S(\vec{\a})\text{ for all }j=1,\dots,r\big]\P[\tilde{\n}_p=q-v_ip]\\
= &\big(1+O(\log^{-16} x)\big)\sigma^r\sum_{i=1}^r\sum_{q,p}
\P[\tilde{\n}_p=q-v_ip]\1_{q-v_ip\in\cE_{p}}.
\end{align*}
Fix $i\in\{1,\ldots,r\}$. For any $1\leq j\leq k_m$ and $q\in\cQ$, assume that $q-v_ip_1\equiv h_j\pmod{p_1}$ and $q-v_ip_2\equiv h_j\pmod{p_2}$ for two distinct primes $p_1,p_2\in\cP$. Then we have $q-h_j'$ is divisible by $p_1p_2$. It is impossible since $q=O(x\log x)$ and $p_1,p_2>x/2$. Hence
$$
\big|\{p\in\cP:\,q-v_ip\not\in\cE_p\}\big|\leq k_m.
$$
It follows from (\ref{sumPqnpvip}) and (\ref{Pnpn}) that
\begin{align}
\sum_{q,p}\label{PnpqvipcEp}
\P[\tilde{\n}_p=q-v_ip)]\1_{q-v_ip\in\cE_{p}}=&
\sum_{q,p}
\P[\tilde{\n}_p=q-v_ip]+O_m\big(|\cQ|\cdot x^{-\frac23+o(1)}\big)\notag\\
=&\bigg(1+O\bigg(\frac{1}{\log^{10}_{2} x}\bigg)\bigg)\cdot\frac{u}{r}\cdot\frac{x}{2y}\cdot|\cQ|,
\end{align}
i.e.,
\begin{align*}
\E\bigg[\sum_{q\in\cQ\cap S({\vec{\a}})}\sum^{r}_{i=1}\sum_{p\in\cP}Z_{p}(\vec{\a};q-v_ip)\bigg]=&\big(1+O(\log_{2}^{-10} x)\big)\sigma^rr\cdot\frac{u}{r}\cdot\frac{x}{2y}\cdot\frac{y}{\log x}\\
=&\big(1+O(\log_{2}^{-10} x)\big)\cdot\frac{\sigma^rux}{2\log x}
\end{align*}
by the prime number theorem.
Similarly,
\begin{align*}
&\E\bigg[\sum_{q\in\cQ\cap S({\vec{\a}})}\bigg(\sum^{r}_{i=1}\sum_{p\in\cP(\vec{\a})}Z_{p}(\vec{\a};q-v_ip)\bigg)^2\bigg]\\
=& \sum_{\substack{i_1,i_2,q\\p_1\neq p_2\\q-v_{i_1}p_1\in\cE_{p_1}\\q-v_{i_2}p_2\in\cE_{p_2}}}\P\big[q+(v_j-v_{i_l})p_l\in S(\vec{\a})\text{ for }1\leq j\leq r,\ l=1,2\big]\prod_{l=1,2}\P[\tilde{\n}^{(l)}_{p_l}=q-v_{i_l}p_l]\\
&+\sum_{\substack{p,q,i_1,i_2\\q-v_{i_1}p\in\cE_{p}\\q-v_{i_2}p\in\cE_{p}}}\P\big[q+(v_j-v_{i})p\in S(\vec{\a})\big]\prod_{l=1,2}\P[\tilde{\n}^{(l)}_{p}=q-v_{i_l}p]\\
=&\big(1+O(\log_{2}^{-10} x)\big)\bigg(\sigma^{2r-1}\sum_{p_1\neq p_2}\sum_{q,i_1,i_2}
\prod_{l=1,2}\P[\tilde{\n}^{(l)}_{p_l}=q-v_{i_l}p_l]\1_{q-v_{i_l}p_l\in\cE_{p}}+O(x^{-\frac{2}{3}+o(1)})\frac{\sigma^rux}{2\log x}\bigg),
\end{align*}
by noting that
$$
\sum_{p}\sum_{q,i_1,i_2}
\prod_{l=1,2}\P[\tilde{\n}^{(l)}_{p}=q-v_{i_l}p]\1_{q-v_{i_l}p\in\cE_{p}}=O(x^{-\frac23+o(1)})\sum_{i}\sum_{q,p}
\P[\tilde{\n}_{p}=q-v_{i}p]\1_{q-v_ip\in\cE_{p}}
$$
by (\ref{Pnpn}). It follows from (\ref{PnpqvipcEp}) that
\begin{align*}
&\sum_{p_1\neq p_2}\sum_{q,i_1,i_2}
\prod_{l=1,2}\P[\tilde{\n}^{(l)}_{p_l}=q-v_{i_l}p_l]\1_{q-v_{i_l}p_l\in\cE_{p}}\\
=&\bigg(\sum_{i}\sum_{q,p}
\P[\tilde{\n}_{p}=q-v_{i}p]\1_{q-v_ip\in\cE_{p}}+O(x^{-\frac23+o(1)})\bigg)^2\\
=&\big(1+O(\log_{2}^{-10} x)\big)\cdot\frac{u^2x^2}{4\log^2 x}
\end{align*}

Let $\E_{\vec{\a}}$ (resp. $\P_{\vec{\a}}$) denote the expectation (resp. probability) with respect to $\vec{\a}$.  In view of (\ref{QjSabound}), with the probability of $1-o(1)$, we have
\begin{align*}
O(\log_2^{-10} x)\cdot\bigg(\frac{u}{\sigma}\cdot\frac{x}{2y}\bigg)^2=&\E_{\vec{\a}}\bigg[\frac{1}{|\cQ\cap S(\vec{\a})|}\sum_{q\in\cQ\cap S(\vec{\a})}\bigg(\sigma^{-r}\sum_{p\in\cP}Z_{p}(\vec{\a};q-v_ip)-\frac{u}{\sigma}\cdot\frac{x}{2y}\bigg)^2\bigg]\\
=&\E_{\vec{\a},\q}\bigg[\bigg(\sigma^{-r}\sum_{p\in\cP}Z_{p}(\vec{\a};q-v_ip)-\frac{u}{\sigma}\cdot\frac{x}{2y}\bigg)^2\bigg]\\
=&\E_{\vec{\a}}\bigg[\E_{\q}\bigg[\bigg(\sigma^{-r}\sum_{p\in\cP}Z_{p}(\vec{\a};q-v_ip)-\frac{u}{\sigma}\cdot\frac{x}{2y}\bigg)^2\bigg]\bigg],
\end{align*}
where $\q$ is chosen uniformly at random from $\cQ\cap S(\vec{\a})$.
By the Markov inequality,
$$
\P_{\vec{\a}}\bigg[\E_{\q}\bigg[\bigg(\sigma^{-r}\sum_{p\in\cP}Z_{p}(\vec{\a};q-v_ip)-\frac{u}{\sigma}\cdot\frac{x}{2y}\bigg)^2\bigg] \geq \log^{-9}_2 x\cdot\bigg(\frac{ux}{2\sigma y}\bigg)^2\bigg]=o\bigg(\frac{1}{\log_3 x}\bigg),
$$
i.e.,
\begin{equation}\label{EFqRa}
\E_{\q}\bigg[\bigg(\sigma^{-r}\sum_{p\in\cP}Z_{p}(\vec{\a};q-v_ip)-\frac{u}{\sigma}\cdot\frac{x}{2y}\bigg)^2\bigg] \leq \log^{-9}_2 x\cdot\bigg(\frac{ux}{2\sigma y}\bigg)^2
\end{equation}
with the probability $1-o(1)$ of $\vec{\a}$.
Applying Markov's inequality again, for almost all $\vec{a}$, we obtain that
$$
\P_{\q}\bigg[\bigg|\sigma^{-r}\sum_{p\in\cP}Z_{p}(\vec{\a};q-v_ip)-\frac{u}{\sigma}\cdot\frac{x}{2y}\bigg|\geq \log^{-3}_2 x\cdot\frac{ux}{2\sigma y}\bigg]=O\big(\log_2^{-3} x\big).
$$
Since the choice of $\q$ is chosen uniformly random, the number of bad $q$ is at most
$$O\bigg(\frac{\sigma y}{\log x}\cdot \frac{1}{\log_2^{3} x}\bigg)$$ with the probability $1-o(1)$ of $\vec{\a}$.
\end{proof}\begin{lemma}\label{FGKMTlem62}
With the probability $1-o(1)$, we have
\begin{equation}\label{sigam-rsumsum}
\sigma^{-r}\sum_{i=1}^{r}\sum_{p\in\cP({\vec{\a}})}Z_p(\vec{\a};q-v_ip)=\big(1+O(\log_2^{-3} x)\big)\cdot\frac{u}{\sigma}\cdot\frac{x}{2y}
\end{equation}
for all but at most $x/(2\log x\log_2 x)$ of the prime $q\in\cQ\cap S(\vec{\a})$.
\end{lemma}
\begin{proof}
Fix $i$ and substitute $n=q-v_ip$. By Lemma \ref{estQjSa} and (\ref{FGKMTpf62620}), with the probability $1-o(1)$, we have
\begin{align*}
o\bigg(\frac{u}{\sigma}\cdot\frac{x}{2y}\cdot\frac{1}{r\log^3_2x\log^2_2 x}\bigg)  =& \E_{\vec{\a}}\bigg[\frac{1}{|\cQ\cap S(\vec{\a})|}\sum_{q\in\cQ\cap S(\vec{\a})}\sigma^{-r}\sum_{p\in\cP\backslash\cP(\vec{\a})}Z_{p}(\vec{\a};q-v_ip)\bigg]\\
=& \E_{\vec{\a}}\bigg[\E_{\q}\bigg[\sigma^{-r}\sum_{p\in\cP\backslash\cP(\vec{\a})}Z_{p}(\vec{\a};\q-v_ip)\bigg]\bigg],
\end{align*}
where we choose those $\q\in\cQ\cap S(\vec{\a})$ uniformly randomly. Using Markov's inequality, we get
$$
\E\bigg[\sigma^{-r}\sum_{p\in\cP\backslash\cP(\vec{\a})}Z_{p}(\vec{a};\q-v_ip)\bigg]\leq \frac{u}{\sigma}\cdot\frac{x}{2y}\cdot\frac{1}{r\log^3_2x\log^2_2 x}$$
with the probability $1-o(1)$ of $\vec{\a}$.
Using the Markov inequality again, we have
$$
\P\bigg[\sigma^{-r}\sum_{p\in\cP\backslash\cP(\vec{\a})}Z_{p}(\vec{a};\q-v_ip)\geq \frac{u}{\sigma}\cdot\frac{x}{2y}\cdot\frac{1}{r\log^3_2x}\bigg] =O\bigg(\frac{1}{\log_2^2 x}\bigg).
$$
Hence the number of those $q\in \cQ\cap S(\vec{\a})$ such that
$$
\sigma^{-r}\sum_{p\in\cP\backslash\cP(\vec{\a})}Z_{p}(\vec{a};\q-v_ip)\geq \frac{u}{\sigma}\cdot\frac{x}{2y}\cdot\frac{1}{r\log^3_2x}
$$
is at most
$$
|\cQ\cap S(\vec{\a})|\cdot O\bigg(\frac{1}{\log_2^2 x}\bigg)=O\bigg(\frac{x}{\log x\log_2 x}\bigg).
$$
Thus by Lemma \ref{sigam-rsumsumP}, (\ref{sigam-rsumsum}) is concluded.
\end{proof}

Now we are ready to complete the proof of Proposition \ref{RanCon}.
Suppose that  $\vec{a}$ is good (the definition from Proposition \ref{RanCon}) and $q\in\cQ\cap S(\vec{a})$. Clearly $q=\n_p + v_ip$ only if $p\in\cP(\vec{a})$, since $\n_p = 0$ for $p\in\cP\setminus \cP(\vec{a})$. Substituting (\ref{PZ}) into
the left hand side of of (\ref{sigam-rsumsum}) and applying (\ref{Xpabound}), we find that
\begin{align*}
\sigma^{-r}\sum_{i=1}^{r}\sum_{p\in\cP(\vec{a})}Z_p(\vec{a};q-v_ip)= &\sigma^{-r}\sum_{i=1}^{r}\sum_{p\in\cP(\vec{a})}X_p(\vec{a})\P\big[\n_p=q-v_ip\big|\,\vec{\a}=\vec{a}\big]\\
= &\big(1+O(\log^{-3} x)\big)\cdot\sigma^{-r}\sum_{i=1}^{r}\sum_{p\in\cP(\vec{a})}\P\big[\n_p=q-v_ip\big|\,\vec{\a}=\vec{a}\big]\\
= &\big(1+O(\log^{-3} x)\big)\cdot\sigma^{-r}\sum_{p\in\cP}\P\big[q\in\e_p(\vec{a})\big|\,\vec{\a}=\vec{a}\big],
\end{align*}
where $$\e_p(\vec{a}) =\{\n_p+v_ip:\,1 \leq i \leq r\} \cap\cQ\cap S(\vec{a})$$ is the one defined in Proposition \ref{RanCon}. Thus in view of (\ref{maxr}), (\ref{SWtau}) and (\ref{Sigma_y}), (III) holds with the probability $1-o(1)$, by setting
$$
C:=\frac{u}{\sigma}\cdot\frac{x}{2y}\asymp\frac{1}{c_0}.
$$

\section{The least prime in an arithmetic progression}
\setcounter{lemma}{0}
\setcounter{corollary}{0}
\setcounter{remark}{0}
\setcounter{equation}{0}
\setcounter{conjecture}{0}
\setcounter{proposition}{0}

For two positive integers $k,l$ with $(k,l)=1$ and $k\geq l$, let $p(k,l)$ denote the least prime in the arithmetic progression $\{kn+l:\,n\geq 1\}$. In \cite{Prachar} (see also \cite{Schinzel}), with help of the Erd\"os-Rankin method, Prachar proved that for any given $l\geq 1$, there exists infinitely many $k$ such that
\begin{equation}
p(k,l)\geq ck\cdot\frac{\log k\log_2k\log_4k}{\log_3^2k},
\end{equation}
where $c>0$ is a absolute constant. 

Here, we may improve the result of Prachar as follows.
\begin{theorem}\label{lpg}
For any given $l\geq 1$, there exists infinitely many $k$ such that
\begin{equation}
p(k,l)\geq ck\cdot\frac{\log k\log_2k\log_4k}{\log_3^2k},
\end{equation}
where $c>0$ is a absolute constant. 
\end{theorem}
Suppose that $x$ is a sufficiently large integer. 
Let $y,\alpha,\cS,\cP,\cQ$ and $S(\vec{a})$ are likewise defined in Sections 2 and 3. 
The only difference is that here we need to set 
$$
\bar{\cE}_p^*:=\{a\mod{p}:\,a\not\equiv x\pmod{p}\}
$$
for any prime $p$. Similarly, we can prove that
\begin{proposition}\label{sievep2}
There are two vectors $\vec{a}=(a_s\in \bar{\cE}_s^*)_{s\in \cS}$ and $\vec{b}=(b_p\in \bar{\cE}_p^*)_{p\in \cP}$, such that
\begin{equation}\label{QjSajSb2}
\big|\cQ\cap S(\vec{a})\cap S(\vec{b})\big|\ll \frac{x}{\log x}.
\end{equation}
\end{proposition}
Let us show how to deduce Theorem \ref{lpg} from Proposition \ref{sievep2}.
Let $\vec{a}$ and $\vec{b}$ be as in Proposition \ref{sievep2}. Extend the tuple $(a_s\in \bar{\cE}_s^*)_{s\in \cS}$  to $(a_s\in \bar{\cE}_s^*)_{p\leq x}$ by setting $a_p:=b_p$ for $p\in \cP$ and $a_p:=0$ for those $p\not\in\cS\cup\cP$.
Let
$$
\cT:=\{ n\in(x,y]:n\not\equiv a_p\pmod{p}\text{ for all }p\leq x \}
$$
According to Proposition \ref{sievep2} and the discussions in Section 3, we also can get that $|\cT|\ll x/\log x$.
Choose  a sufficiently large constant $C_0$ such that $|\cT|$ is less than the number of primes in $(x,C_0x]$. 

It is not difficult to see that for any $n\in(x,y]$, $n-x$ has at most one prime factor lying in $(x,C_0x]$. So assuming that $\cT=\{t_1,\ldots,t_h\}$, we may choose distinct primes $p_1,\ldots,p_h\in(x,C_0x]$ such that $t_i-x\not\equiv 0\pmod{p_i}$ for each $1\leq i\leq h$. Let $a_{p_i}=t_i$ for $1\leq i\leq h$ and arbitrarily choose $a_p\in\cE_p^*$ for those $p\in(x,C_0x]\setminus\{p_1,\ldots,p_h\}$. Thus for each $n\in(x,y]$, there exists $p\leq C_0x$ such that $n\equiv a_p\pmod{p}$.

Now let $x=l^\beta$ where $\beta$ is an arbitrary sufficiently large integer. Let
$$
M:=\prod_{p\leq C_0x}p.
$$  
By the Chinese remainder theorem, there exists $k\in(C_0x-l,M+C_0x-l]$ such that
$$
k(a_p-x)\equiv -l\pmod{p}
$$
for each $p\leq C_0x$ with $p\nmid l$, and $k\equiv 1\pmod{p}$ if $p\mid l$. Clearly $$k+l>C_0x$$ and $$\log k\ll\log M\ll x$$ by the prime number theorem.
For any $1\leq n\leq y-x$, there exists $p\leq C_0x$ such that $n+x\equiv a_p\pmod{p}$. If $p\nmid l$, then
$$
kn+l\equiv k(a_p-x)+l\equiv 0\pmod{p}.
$$
Suppose that $p\mid l$. Recall that $a_p=0$ since $l<\log^{20} x$. Since $l$ divides $x$, we also have
$$
kn+l\equiv -kx+l\equiv 0\pmod{p}.
$$
So the elements of $\{kn+l:\,1\leq n\leq y-x\}$ are all composite. Thus we get the desired result by recalling that
$$
y-x\gg x\cdot\frac{\log x\log_3 x}{\log_2 x}.
$$
\qed

     \end{document}